\documentclass[10pt, a4paper]{article}

\usepackage[pass, textheight=538pt,textwidth=318pt]{geometry}

\usepackage{mathptmx}
\usepackage{amsthm, amssymb, amsmath}
\usepackage{pstricks,xcolor}
\usepackage{enumerate}
\usepackage{authblk}
\usepackage{url}

\title{Boxicity and Cubicity of Product Graphs}
\date{}

\author[1]{L.~Sunil~Chandran}
\author[2]{Wilfried~Imrich}
\author[3]{Rogers~Mathew \footnote{Supported by an AARMS Postdoctoral Fellowship}}
\author[1]{Deepak~Rajendraprasad \footnote{Supported by Microsoft Research India PhD Fellowship}}

\affil[1]{
	Department of Computer Science and Automation, \authorcr 
	Indian Institute of Science, Bangalore, India - 560012. \authorcr
	\{sunil, deepakr\}@csa.iisc.ernet.in
}
\affil[2]{
	Department Mathematics and Information Technology, \authorcr
	Montanuniversit\"{a}t Leoben, Austria. \authorcr
	imrich@unileoben.ac.at 
}
\affil[3]
{
	Department of Mathematics and Statistics, \authorcr 
	Dalhousie University, 
	Halifax, Canada - B3H 3J5. \authorcr
	rogersm@mathstat.dal.ca
}

\theoremstyle{definition}
\newtheorem{definition}{Definition}

\theoremstyle{plain}
\newtheorem{theorem}{Theorem}
\newtheorem{lemma}[theorem]{Lemma}
\newtheorem{corollary}[theorem]{Corollary}

\newtheorem{observation}{Observation}

\theoremstyle{remark}
\newtheorem*{remark}{Remark}

\newtheoremstyle{tightness}
  {}
  {}
  {\itshape}
  {}
  {}
  {.}
  {0.5em}
  {{\bfseries \thmnote{#3}}}
\theoremstyle{tightness}

\newtheoremstyle{restatement}
  {}
  {}
  {}
  {}
  {}
  {.}
  {0.5em}
  {{\bfseries \thmnote{#3}}\medskip \\{\itshape Statement}}
\theoremstyle{restatement}
\newtheorem*{restatement}{Statement}

\newcommand{\floor}[1]{\lfloor #1 \rfloor}
\newcommand{\ceil}[1]{\lceil #1 \rceil}
\newcommand{\AND}{\textnormal{ and }}
\newcommand{\OR}{\textnormal{ or }}
\newcommand{\IF}{\textnormal{ if }}

\newcommand{\order}[1]{O\left( #1 \right)}
\newcommand{\orderatleast}[1]{\Omega\left( #1 \right)}
\newcommand{\orderexactly}[1]{\Theta\left( #1 \right)}

\def\half{\frac{1}{2}}

\def\into{\rightarrow}
\def\implies{\Rightarrow}

\def\R{\mathbb{R}}
\def\Z{\mathbb{Z}}

\DeclareMathOperator{\boxicity}{boxicity}
\DeclareMathOperator{\cubicity}{cubicity}
\DeclareMathOperator{\hdim}{dim}
\DeclareMathOperator{\pdim}{pdim}
\DeclareMathOperator{\cart}{\Box}
\DeclareMathOperator{\strong}{\boxtimes}
\DeclareMathOperator{\direct}{\times}
\DeclareMathOperator{\sdiff}{\triangle}

\def\calF{\mathcal{F}}
\def\calD{\mathcal{D}}
\def\calP{\mathcal{P}}
\def\Gcd{G^{\cart d}}
\def\Gsd{G^{\strong d}}
\def\Gdd{G^{\direct d}}

\begin{document}
\maketitle
\begin{abstract}
    
The {\em boxicity} ({\em cubicity}) of a graph $G$ is the minimum natural number $k$ such that $G$ can be represented as an intersection graph of axis-parallel rectangular boxes (axis-parallel unit cubes) in $\R^k$. In this article, we give estimates on the boxicity and the cubicity of {\em Cartesian}, {\em strong} and {\em direct products} of graphs in terms of invariants of the component graphs. In particular, we study the growth, as a function of $d$, of the boxicity and the cubicity of the $d$-th power of a graph with respect to the three products. Among others, we show a surprising result that the boxicity and the cubicity of the $d$-th Cartesian power of any given finite graph is in $\order{\log d / \log\log d}$ and $\orderexactly{d / \log d}$, respectively. On the other hand, we show that there cannot exist any sublinear bound on the growth of the boxicity of powers of a general graph with respect to strong and direct products.

\vspace{1ex} \noindent \textbf{Keywords:} intersection graphs, boxicity, cubicity, graph products, boolean lattice.
\end{abstract}

\section{Introduction}

Throughout this discussion, a {\em $k$-box} is the Cartesian product of $k$ closed intervals on the real line $\R$, and a {\em $k$-cube} is the Cartesian product of $k$ closed unit length intervals on $\R$. Hence both are subsets of $\R^k$ with edges parallel to one of the coordinate axes. All the graphs considered here are finite, undirected and simple. 

\begin{definition}[Boxicity, Cubicity]
\label{definitionBoxicityCubicity}
A {\em $k$-box representation} ({\em $k$-cube representation}) of a graph $G$ is a function $f$ that maps each vertex of $G$ to a $k$-box ($k$-cube) such that for any two distinct vertices $u$ and $v$ of $G$, the pair $uv$ is an edge in $G$ if and only if the boxes $f(u)$ and $f(v)$ have a non-empty intersection. The {\em boxicity} ({\em cubicity}) of a graph $G$, denoted by $\boxicity(G)$ ($\cubicity(G)$), is the smallest natural number $k$ such that $G$ has a $k$-box ($k$-cube) representation. 
\end{definition}

It follows from the above definition that complete graphs have boxicity and cubicity $0$ and interval graphs (unit interval graphs) are precisely the graphs with boxicity (cubicity) at most $1$. The concepts of boxicity and cubicity were introduced by F.S. Roberts in 1969 \cite{Roberts}. He showed that every graph on $n$ vertices has an $\floor{n/2}$-box and a $\floor{2n/3}$-cube representation. 

Given two graphs $G_1$ and $G_2$ with respective box representations $f_1$ and $f_2$, let $G$ denote the graph on the vertex set $V(G_1) \times V(G_2)$ whose box representation is a function $f$ defined by $f((v_1, v_2)) = f_1(v_1) \times f_2(v_2)$. It is not difficult to see that $G$ is the usual strong product of $G_1$ and $G_2$ (cf. Definition \ref{definitionGraphProducts}). Hence it follows that the boxicity (cubicity) of $G$ is at most the sum of the boxicities (cubicities) of $G_1$ and $G_2$. The interesting question here is: {\em can it be smaller?} We show that {\em it can be smaller} in general. But in the case when $G_1$ and $G_2$ have at least one universal vertex each, we show that that the boxicity (cubicity) of $G$ is equal to the sum of the boxicities (cubicities) of $G_1$ and $G_2$ (Theorem \ref{theoremStrongProduct}).

\begin{definition}[Graph products]
\label{definitionGraphProducts}
The {\em strong product}, the {\em Cartesian product} and the {\em direct product} of two graphs $G_1$ and $G_2$, denoted respectively by $G_1 \strong G_2$, $G_1 \cart G_2$ and $G_1 \direct G_2$, are graphs on the vertex set $V(G_1) \times V(G_2)$ with the following edge sets:
\[
\begin{array}{rcl}
E(G_1 \strong G_2) &=& \{(u_1,u_2)(v_1,v_2) : 
	(u_1 = v_1 \OR u_1v_1 \in E(G_1)) \AND \\
	& & 	(u_2 = v_2 \OR u_2v_2 \in E(G_2)) \}, \\
E(G_1 \cart G_2) &=& \{(u_1,u_2)(v_1,v_2) : 
	(u_1 = v_1, u_2v_2 \in E(G_2)) \OR \\
	& & 	(u_1v_1 \in E(G_1), u_2 = v_2) \}, \\
E(G_1 \direct G_2) &=& \{(u_1,u_2)(v_1,v_2) : 
	u_1v_1 \in E(G_1) \AND
	u_2v_2 \in E(G_2) \}. 
\end{array}
\]
The {\em $d$-th strong power}, {\em Cartesian power} and {\em direct power} of a graph $G$ with respect to each of these products, that is, the respective product of $d$ copies of $G$, are denoted by $\Gsd$, $\Gcd$ and $\Gdd$, respectively. Please refer to \cite{imrich2011handbook} to know more about graph products. 
\end{definition}

Unlike the case in strong product, the boxicity (cubicity) of the Cartesian and direct products can have a boxicity (cubicity) larger than the sum of the individual boxicities (cubicities). For example, while the complete graph on $n$ vertices $K_n$ has boxicity $0$, we show that the Cartesian product of two copies of $K_n$ has boxicity at least $\log n$ and the direct product of two copies of $K_n$ has boxicity at least $n-2$. In this note, we give estimates on boxicity and cubicity of Cartesian and direct products in terms of the boxicities (cubicities) and chromatic number of the component graphs. This answers a question raised by Douglas B. West in 2009 \cite{west2008boxicity}.  

We also study the growth, as a function of $d$, of the boxicity and the cubicity of the $d$-th power of a graph with respect to these three products. Among others, we show a surprising result that the boxicity and the cubicity of the $d$-th Cartesian power of any given finite graph is in $\order{\log d / \log\log d}$ and $\orderexactly{d / \log d}$, respectively (Corollary \ref{corollaryGrowthCartesian}). To get this result, we had to obtain non-trivial estimates on boxicity and cubicity of hypercubes and Hamming graphs and a bound on boxicity and cubicity of the Cartesian product which does not involve the sum of the boxicities or cubicities of the component graphs.

The results are summarised in the next section after a brief note on notations. The proofs and figures are moved to the appendix in the interest of space. 

\subsection{Notational note}
The vertex set and edge set of a graph $G$ are denoted, respectively, by $V(G)$ and $E(G)$. A pair of distinct vertices $u$ and $v$ is denoted at times by $uv$ instead of $\{u,v\}$ in order to avoid clutter. A vertex in a graph is {\em universal} if it is adjacent to every other vertex in the graph. If $S$ is a subset of vertices of a graph $G$, the subgraph of $G$ induced on the vertex set $S$ is denoted by $G[S]$. If $A$ and $B$ are sets, then $A \sdiff B$ denotes their symmetric difference and $A \times B$ denotes their Cartesian product. The set $\{1, \ldots, n\}$ is denoted by $[n]$. All logarithms mentioned are to the base $2$. 

\section{Our Results} 
\label{sectionOurResults}

\subsection{Strong products}

\begin{theorem}
\label{theoremStrongProduct}
Let $G_i$, $i \in [d]$, be graphs with $\boxicity(G_i) = b_i$ and $\cubicity(G_i) = c_i$. Then
\[
\begin{array}{rcccl}
\max_{i=1}^{d} b_i & \leq & 
	\boxicity(\strong_{i=1}^{d} G_i)  & \leq &
	\sum_{i=1}^{d} b_i, \AND \\ 
\max_{i=1}^{d} c_i & \leq & 
	\cubicity(\strong_{i=1}^{d} G_i)  & \leq &
	\sum_{i=1}^{d} c_i. 
\end{array}
\]
Furthermore, if each $G_i$, $i \in [d]$ has a universal vertex, then the second inequality in both the above chains is tight.
\end{theorem}

If we consider the strong product of a $4$-cycle $C_4$ with a path on $3$ vertices $P_3$, we get an example where the upper bound in Theorem \ref{theoremStrongProduct} is not tight. 
Theorem \ref{theoremStrongProduct} has the following interesting corollary.

\begin{corollary}
\label{corollaryGrowthStrong}
For any given graph $G$, $boxicity(\Gsd)$ and $\cubicity(\Gsd)$ are in $\order{d}$ and there exist graphs for which they are in $\orderatleast{d}$.
\end{corollary}

\subsection{Cartesian products}

We show two different upper bounds on the boxicity and cubicity of Cartesian products. The first and the easier result bounds from above the boxicity (cubicity) of a Cartesian product in terms of the boxicity (cubicity) of the corresponding strong product and the boxicity (cubicity) of a Hamming graph whose size is determined by the chromatic number of the component graphs. The second bound is in terms of the maximum cubicity among the component graphs and the boxicity (cubicity) of a Hamming graph whose size is determined by the sizes of the component graphs. The second bound is much more useful to study the growth of boxicity and cubicity of higher Cartesian powers since the first term remains a constant.
 
\begin{theorem}
\label{theoremCartesianStrong}
For graphs $G_1, \ldots, G_d$, 
\begin{eqnarray*}
\boxicity(\cart_{i=1}^d G_i) & \leq & 
	\boxicity(\strong_{i=1}^d G_i) + 
	\boxicity(\cart_{i=1}^d K_{\chi_i})
	\textnormal{ and} \\
\cubicity(\cart_{i=1}^d G_i) & \leq & 
	\cubicity(\strong_{i=1}^d G_i) + 
	\cubicity(\cart_{i=1}^d K_{\chi_i})
\end{eqnarray*}
where $\chi_i$ denotes the chromatic number of $G_i, i \in [d]$.
\end{theorem}

When $G_i = K_q$ for every $i \in [d]$, $G = \strong_{i=1}^d G_i$ is a complete graph on $q^d$ vertices and hence has boxicity and cubicity $0$. In this case it is easy to see that both the bounds in Theorem \ref{theoremCartesianStrong} are tight.

\begin{theorem}
\label{theoremMaxCubicity}
For graphs $G_1, \ldots, G_d$, with $|V(G_i)| = q_i$ and $\cubicity(G_i) = c_i$, for each $i \in [d]$, 
\[
\begin{array}{rcl}
\boxicity(\cart_{i=1}^d G_i)  & \leq &	
	\max_{i \in [d]} c_i  + \boxicity(\cart_{i=1}^d K_{q_i}), 
	\AND \\
\cubicity(\cart_{i=1}^d G_i)  & \leq &	
	\max_{i \in [d]} c_i  + \cubicity(\cart_{i=1}^d K_{q_i}).
\end{array}
\]
\end{theorem}

In wake of the two results above, it becomes important to have a good upper bound on the boxicity and the cubicity of Hamming graphs. The {\em Hamming graph $K_q^d$} is the Cartesian product of $d$ copies of a complete graph on $q$ vertices. We call the $K_2^d$ the {\em $d$-dimensional hypercube}.

The cubicity of hypercubes is known to be in  $\orderexactly{ \frac{d}{\log d} }$. The lower bound is due to Chandran, Mannino and Oriolo \cite{CMO} and the upper bound is due to Chandran and Sivadasan \cite{CN99}. But we do not have such tight estimates on the boxicity of hypercubes. The only explicitly known upper bound is one of $\order{d / \log d}$ which follows from the bound on cubicity since boxicity is bounded above by cubicity for all graphs. The only non-trivial lower bound is one of $\half (\ceil{\log\log d} + 1)$ due to Chandran, Mathew and Sivadasan \cite{RogSunSiv}. 

We make use of a non-trivial upper bound shown by Kostochka on the dimension of the partially ordered set (poset) formed by two neighbouring levels of a Boolean lattice \cite{kostochka1997dimension} and a connection between boxicity and poset dimension established by Adiga, Bhowmick and Chandran in \cite{DiptAdiga} to obtain the following result.

\begin{theorem}
\label{theoremHypercubeBooleanLattice}
Let $b_d$ be the largest dimension possible of a poset formed by two adjacent levels of a Boolean lattice over a universe of $d$ elements. Then
\[
\begin{array}{rcccl}
\half b_d &\leq& 
	\boxicity (K_2^d) &\leq& 
		3b_d.
\end{array}
\]
Furthermore, $\boxicity(K_2^d) \leq 12 \log d / \log\log d$.
\end{theorem}

We would also like to remark that a better upper or lower bound on the boxicity of hypercubes will in turn give a commensurate upper or lower bound on the dimension of the poset formed by neighbouring levels of Boolean lattices.  

In order to extend these results on hypercubes to Hamming graphs, we use multiple weak homomorphisms of the Hamming graph $K_q^d$ into the hypercube $K_2^d$. The homomorphisms are generated based on a labelling of the vertices of each copy of $K_q$ using a double distinguishing family of subsets of a small universe.  A family $\calD$ of sets is called {\em double distinguishing} if for any two pairs of set $A,A'$ and $B, B'$ from $\calD$, such that $A \neq A'$ and $B \neq B'$, we have $ (A \sdiff  A') \cap (B \sdiff B') \neq \emptyset$. The existence of such a family over a small universe is established using probabilistic arguments. This gives us the upper bounds in the following result. The lower bounds follow from a result on boxicity of line graphs of complete bipartite graphs in \cite{basavaraju2012pairwise} once we note that $K_q^2$ is isomorphic to the line graph of a complete bipartite graph.

\begin{theorem}
\label{theoremHamming}
Let $K_q^d$ be the d-dimensional Hamming graph on the alphabet $[q]$ and let $K_2^d$ be the d-dimensional hypercube. Then for $d \geq 2$,
\[ 
\begin{array}{rcccl}
\log q &\leq & 
	boxicity(K_q^d) & \leq & 
	\ceil{10 \log q} \boxicity(K_2^d), \AND \\
\log q &\leq & 
	cubicity(K_q^d) & \leq & 
	\ceil{10 \log q} \cubicity(K_2^d).
\end{array}
\]
\end{theorem}

Theorem \ref{theoremMaxCubicity}, along with the bounds on boxicity and cubicity of Hamming graphs, gives the following corollary which is the main result in this article. The lower bound on the order of growth is due to the presence of $K_2^d$ as an induced subgraph in the $d$-the Cartesian power of any non-trivial graph.

\begin{corollary}
\label{corollaryGrowthCartesian}
For any given graph $G$ with at least one edge, 
\begin{eqnarray*}
\boxicity(\Gcd) &\in& \order{\log d / \log\log d} \cap \orderatleast{\log\log d}, 
	\AND \\
\cubicity(\Gcd) &\in& \orderexactly{d / \log d}.
\end{eqnarray*}
\end{corollary}

\subsection{Direct products}

\begin{theorem}
\label{theoremDirectStrong}
For graphs $G_1, \ldots, G_d$, 
\begin{eqnarray*}
\boxicity(\direct_{i=1}^d G_i) & \leq & 
	\boxicity(\strong_{i=1}^d G_i) + 
	\boxicity(\direct_{i=1}^d K_{\chi_i})
	\textnormal{ and} \\
\cubicity(\direct_{i=1}^d G_i) & \leq & 
	\cubicity(\strong_{i=1}^d G_i) + 
	\cubicity(\direct_{i=1}^d K_{\chi_i})
\end{eqnarray*}
where $\chi_i$ denotes the chromatic number of $G_i, i \in [d]$.
\end{theorem}

In the wake of Theorem \ref{theoremDirectStrong}, it is useful to estimate the boxicity and the cubicity of the direct product of complete graphs. Before stating our result on the same, we would like to discuss a few special cases. If $G = \direct_{i=1}^d K_2$ then $G$ is a perfect matching on $2^d$ vertices and hence has boxicity and cubicity equal to $1$. If $G = K_q \direct K_2$, then it is isomorphic to a graph obtained by removing a perfect matching from the complete bipartite graph with $q$ vertices on each part. This is known as the {\em crown graph} and its boxicity is known to be $\lceil q/2 \rceil$ \cite{chintan}. 

\begin{theorem}
\label{theoremBoxicityCompleteDirect}
Let $q_i \geq 2$ for each $i \in [d]$. Then,
\[
\begin{array}{rcccl}
\frac{1}{2} \sum_{i=1}^d (q_i - 2) &\leq& 
	\boxicity\left(\direct_{i=1}^d K_{q_i} \right) &\leq& 
	\sum_{i=1}^d q_i, \AND \\
\frac{1}{2} \sum_{i=1}^d (q_i - 2) &\leq& 
	\cubicity\left(\direct_{i=1}^d K_{q_i} \right) &\leq& 
	\sum_{i=1}^d q_i \log (n/q_i),
\end{array}
\]
where $n = \Pi_{i=1}^d q_i$ is the number of vertices in $\direct_{i=1}^d K_{q_i}$.
\end{theorem}

We believe that it might be possible to improve the upper bound on cubicity to match its lower bound (up to constants). But we leave it for the future. The two results established above have the following two corollaries.

\begin{corollary}
\label{corollaryDirectBoxicity}
For graphs $G_1, \ldots, G_d$, 
\[
\boxicity(\direct_{i=1}^d G_i) \leq \sum_{i=1}^d(\boxicity(G_i) + \chi(G_i)).
\]
\end{corollary}

\begin{corollary}
\label{corollaryGrowthSDirect}
For any given graph $G$, $boxicity(\Gdd)$ is in $\order{d}$ and there exist graphs for which it is in $\orderatleast{d}$.
\end{corollary}

\bibliographystyle{plain}

\clearpage

\begin{appendix}

\def\thesection{Appendix \Alph{section}}
\section{Preliminaries}
\def\thesection{\Alph{section}}

Before giving the proofs of the results stated in Section \ref{sectionOurResults}, we collect together some results from literature and some easy observations which are used in the proofs given in Section \ref{sectionProofs}. First we give a more combinatorial characterisation for boxicity and cubicity, which is easier to work with at times. 

From Definition \ref{definitionBoxicityCubicity}, it is clear that interval graphs are precisely the graphs with boxicity at most $1$. Given a $k$-box representation of a graph $G$, orthogonally projecting the $k$-boxes to each of the $k$ axes in $\R^k$  gives $k$ families of intervals. Each one of these families can be thought of as an interval representation of some interval graph. Thus we get $k$ interval graphs. It is not difficult to observe that a pair of vertices is  adjacent in $G$ if and only if the pair is adjacent in each of the $k$ interval graphs obtained. Similarly unit interval graphs are precisely the graphs with cubicity $1$, and the orthogonal projections of a $k$-cube representation of a graph $G$ to each of the $k$ axes in $\R^k$ give rise to $k$ unit interval graphs, whose intersection is $G$.

The following lemma, due to Roberts \cite{Roberts}, formalises this relation between box representations and interval graphs.

\begin{lemma}[Roberts \cite{Roberts}]
\label{lemmaRoberts}
For every graph $G$, $\boxicity G \leq k ~ (\cubicity G \leq k)$ if and only if there exist $k$ interval graphs (unit interval graphs) $I_1, \ldots, I_k$, with $V(I_1) = \cdots = V(I_k) = V(G)$ such that $G = I_1 \cap \cdots \cap I_k$.
\end{lemma}
From the above lemma, we get these alternate definitions of boxicity and cubicity. 

\begin{definition}
\label{definitionBoxicityInterval}
The {\em boxicity} ({\em cubicity}) of a graph $G$ is the minimum positive integer $k$ for which there exist $k$ interval graphs (unit interval graphs) $I_1,\ldots, I_k$ such that $G = I_1 \cap \cdots \cap I_k$.  
\end{definition}

Note that if $G = I_1 \cap \cdots \cap I_k$, then each $I_i$ is a supergraph of $G$. Moreover, for every pair of vertices $u,v \in V(G)$ with $\{u,v\} \notin E(G)$, there exists some $i \in [k]$ such that $\{u,v\} \notin E(I_i)$. Hence finding a $k$-box representation ($k$-cube representation) of a graph $G$ is the same as finding $k$ interval supergraphs (unit interval supergraphs) of $G$ with the property that every pair of non-adjacent vertices in $G$ is non-adjacent in at least one of those supergraphs. The following observations are immediate from one of the definitions of boxicity and cubicity.

\begin{observation}
\label{observationInducedSubgraph}
If $H$ is an induced subgraph of a graph $G$, then the boxicity (cubicity) of $H$ is at most the boxicity (cubicity) of $G$.
\end{observation}

\begin{observation}
\label{observationIntersection}
The {\em intersection} of two graphs $G_1$ and $G_2$ on the same vertex set is the graph, denoted by $G_1 \cap G_2$, is the graph on the same vertex set with edge set $E(G_1) \cap E(G_2)$. The boxicity (cubicity) of $G_1 \cap G_2$ is at most the sum of the boxicities (cubicities) of $G_1$ and $G_2$.
\end{observation}

\begin{observation}
\label{observationDisjointUnion}
The {\em disjoint union} of two graphs $G_1$ and $G_2$ on disjoint vertex sets, denoted by $G_1 \uplus G_2$, is the graph with vertex set $V(G_1) \cup V(G_2)$ and edge set $E(G_1) \cup E(G_2)$. The boxicity (cubicity) of $G_1 \uplus G_2$ is equal to the larger of the boxicities (cubicities) of $G_1$ and $G_2$.
\end{observation}

The next observation is not as easy, but follows once we note that, since interval graphs cannot contain induced $4$-cycles, in any interval supergraph of $G_1 \otimes G_2$ either $V(G_1)$ or $V(G_2)$ induces a complete graph. 

\begin{observation}
\label{observationJoin}
The {\em join} of two graphs $G_1$ and $G_2$ on disjoint vertex sets, denoted by $G_1 \otimes G_2$, is the graph with vertex set $V(G_1) \cup V(G_2)$ and edge set $E(G_1) \cup E(G_2) \cup \{\{u_1, u_2\} : u_1 \in V(G_1), u_2 \in V(G_2)\}$. Then
\[
\begin{array}{rcll}
	\boxicity(G_1 \otimes G_2) 	&=& \boxicity G_1 + \boxicity G_2, & \textnormal{and }\\ 
	\cubicity(G_1 \otimes G_2) 	&\geq& \cubicity G_1 + \cubicity G_2.
\end{array}
\]
\end{observation}

\begin{observation}
\label{observationUniversalVertices}
Let $G$ be a graph and $S$ be a set of vertices outside $V(G)$. Then $G \otimes S$ denotes the join of $G$ and a complete graph on $S$, that is, $V(G \otimes S) = V(G) \cup S$ and $E(G \otimes S) = E(G) \cup \{\{v,s\} : v \in V(G) \cup S, s \in S \}$. The boxicity of $G \otimes S$ is equal to the boxicity of $G$.
\end{observation}


\begin{observation}
\label{observationClaw}
A {\em star graph} $S_n$ with {\em root} $r$ is the graph with the vertex set $\{r\} \cup [n]$ and edge set $\{\{r,l\} : l \in [n]\}$. The cubicity of $S_n$ is $\ceil{ \log n }$ while its boxicity is $1$.
\end{observation}

\def\thesection{Appendix \Alph{section}}
\section{Proofs}
\def\thesection{\Alph{section}}
\label{sectionProofs}

\subsection{Strong products}

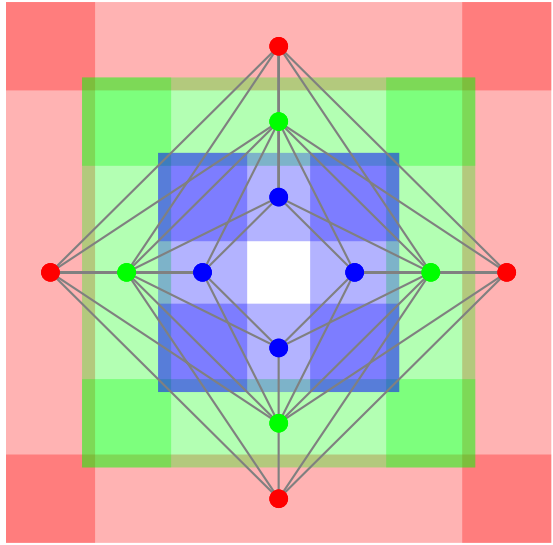
\begin{figure}[h]
\begin{center}
\psset{xunit=0.5cm}
\psset{yunit=0.5cm}

\begin{pspicture}(0,0)(16,16)
	
	\psframe[fillstyle=solid,opacity=0.3,linestyle=none,framearc=0,fillcolor=red]%
		(0.8,0.8)(3.2,15.2)
	\psframe[fillstyle=solid,opacity=0.3,linestyle=none,framearc=0,fillcolor=red]%
		(0.8,0.8)(15.2,3.2)
	\psframe[fillstyle=solid,opacity=0.3,linestyle=none,framearc=0,fillcolor=red]%
		(12.8,0.8)(15.2,15.2)
	\psframe[fillstyle=solid,opacity=0.3,linestyle=none,framearc=0,fillcolor=red]%
		(0.8,12.8)(15.2,15.2)

	\psframe[fillstyle=solid,opacity=0.3,linestyle=none,framearc=0,fillcolor=green]%
		(2.8,2.8)(5.2,13.2)
	\psframe[fillstyle=solid,opacity=0.3,linestyle=none,framearc=0,fillcolor=green]%
		(2.8,2.8)(13.2,5.2)
	\psframe[fillstyle=solid,opacity=0.3,linestyle=none,framearc=0,fillcolor=green]%
		(10.8,2.8)(13.2,13.2)
	\psframe[fillstyle=solid,opacity=0.3,linestyle=none,framearc=0,fillcolor=green]%
		(2.8,10.8)(13.2,13.2)

	\psframe[fillstyle=solid,opacity=0.3,linestyle=none,framearc=0,fillcolor=blue]%
		(4.8,4.8)(7.2,11.2)
	\psframe[fillstyle=solid,opacity=0.3,linestyle=none,framearc=0,fillcolor=blue]%
		(4.8,4.8)(11.2,7.2)
	\psframe[fillstyle=solid,opacity=0.3,linestyle=none,framearc=0,fillcolor=blue]%
		(8.8,4.8)(11.2,11.2)
	\psframe[fillstyle=solid,opacity=0.3,linestyle=none,framearc=0,fillcolor=blue]%
		(4.8,8.8)(11.2,11.2)

	\pspolygon[linecolor=gray](8,2)(2,8)(8,14)(14,8)
	\pspolygon[linecolor=gray](8,4)(4,8)(8,12)(12,8)
	\pspolygon[linecolor=gray](8,6)(6,8)(8,10)(10,8)
	\psline[linecolor=gray]
	(8,2)(8,4)(2,8)(4,8)(8,14)(8,12)(14,8)(12,8)(8,2)(4,8)(2,8)(8,12)(8,14)(12,8)(14,8)(8,4)
	\psline[linecolor=gray]
	(8,6)(8,4)(6,8)(4,8)(8,10)(8,12)(10,8)(12,8)(8,6)(4,8)(6,8)(8,12)(8,10)(12,8)(10,8)(8,4)

	\psdots[dotstyle=o,dotsize=7pt,fillstyle=solid,linecolor=red,fillcolor=red]%
		(8,2)(2,8)(8,14)(14,8)
	\psdots[dotstyle=o,dotsize=7pt,fillstyle=solid,linecolor=green,fillcolor=green]%
		(8,4)(4,8)(8,12)(12,8)
	\psdots[dotstyle=o,dotsize=7pt,fillstyle=solid,linecolor=blue,fillcolor=blue]%
		(8,6)(6,8)(8,10)(10,8)

\end{pspicture}
\end{center}
\caption{The graph $C_4 \strong P_3$ and its $2$-box representation. Every box represents the vertex of the same colour at its center.}
\label{figureC4P3}
\end{figure}
\begin{restatement}[Proof of Theorem \ref{theoremStrongProduct}]
Let $G_i$, $i \in [d]$, be graphs with $\boxicity(G_i) = b_i$ and $\cubicity(G_i) = c_i$. Then
\[
\begin{array}{rcccl}
\max_{i=1}^{d} b_i & \leq & 
	\boxicity(\strong_{i=1}^{d} G_i)  & \leq &
	\sum_{i=1}^{d} b_i, \AND \\ 
\max_{i=1}^{d} c_i & \leq & 
	\cubicity(\strong_{i=1}^{d} G_i)  & \leq &
	\sum_{i=1}^{d} c_i. 
\end{array}
\]
Furthermore, if each $G_i$, $i \in [d]$ has a universal vertex, then the second inequality in both the above chains is tight.
\end{restatement}
\begin{proof}
The lower bounds follow easily since the component graphs are present as induced subgraphs in the product. Let $G = \strong_{i=1}^d G_i$ and $b = \sum_{i=1}^d b_i$. Furthermore, let $f_i$ be a $b_i$-box representation of $G_i$, $i \in [d]$. It is easy to see that $f$ defined by $f((v_1, \ldots, v_d)) = f(v_1) \times \cdots \times f(v_d)$, (where $\times$ denotes the Cartesian product) is a $b$-box representation for $G$. The case for cubicity is also similar.

Let $u_i$ be a universal vertex of $G_i$ for each $i \in [d]$. Now for each $i \in [d]$, set $A_i = \{(a_1, \ldots, a_d) \in V(G) : a_i \in V(G_i) \AND a_j = u_j \IF j \neq i \}$ so that $G[A_i]$ is isomorphic to $G_i$. Since interval graphs do not contain induced $4$-cycles, in any interval supergraph of $G$, all but at most one set among $A_i$, $i \in [d]$, must induce a complete graph. Hence the boxicity (cubicity) of $G$ is at least $\sum_{i=1}^d b_i$ ($\sum_{i=1}^d c_i$).
\end{proof}
If we consider the strong product of a $4$-cycle $C_4$ with a path on $3$ vertices $P_3$, we get an example where the upper bound in Theorem \ref{theoremStrongProduct} is not tight. It is easy to check that $\boxicity(C_4) = 2$ and $\boxicity(P_3) = 1$. Figure \ref{figureC4P3} shows a $2$-box representation of $C_4 \strong P_3$.

\subsection{Cartesian products}

\begin{restatement}[Proof of Theorem \ref{theoremCartesianStrong}]
For graphs $G_1, \ldots, G_d$, 
\begin{eqnarray*}
\boxicity(\cart_{i=1}^d G_i) & \leq & 
	\boxicity(\strong_{i=1}^d G_i) + 
	\boxicity(\cart_{i=1}^d K_{\chi_i})
	\textnormal{ and} \\
\cubicity(\cart_{i=1}^d G_i) & \leq & 
	\cubicity(\strong_{i=1}^d G_i) + 
	\cubicity(\cart_{i=1}^d K_{\chi_i})
\end{eqnarray*}
where $\chi_i$ denotes the chromatic number of $G_i, i \in [d]$.
\end{restatement}
\begin{proof}
Let $G_{\cart} = \cart_{i=1}^d G_i$, $G_{\strong} = \strong_{i=1}^d G_i$ and $K_{\cart} = \cart_{i=1}^d K_{\chi_i}$. Let $b_s = \boxicity(G_{\strong})$ and $b_{\chi} = \boxicity(K_{\cart})$. Furthermore, let $f_s$ and $f_{\chi}$ be $b_s$-box and $b_{\chi}$-box representations of $G_{\strong}$ and $K_{\cart}$, respectively. Finally, let $c_i : V(G_i) \into [\chi_i]$ be a proper colouring of $G_i, i \in [d]$. It is easy to see that $f$ defined by $f((v_1, \ldots, v_d)) = f_s((v_1, \ldots, v_d)) \times f_{\chi}((c_1(v_1), \ldots, c_d(v_d)))$, is a $(b_s + b_{\chi})$-box representation for $G_{\cart}$. The case for cubicity is also similar.
\end{proof}

\begin{restatement}[Proof of Theorem \ref{theoremMaxCubicity}]
For graphs $G_1, \ldots, G_d$, with $|V(G_i)| = q_i$ and $\cubicity(G_i) = c_i$, for each $i \in [d]$, 
\[
\begin{array}{rcl}
\boxicity(\cart_{i=1}^d G_i)  & \leq &	
	\max_{i \in [d]} c_i  + \boxicity(\cart_{i=1}^d K_{q_i}), 
	\AND \\
\cubicity(\cart_{i=1}^d G_i)  & \leq &	
	\max_{i \in [d]} c_i  + \cubicity(\cart_{i=1}^d K_{q_i}).
\end{array}
\]
\end{restatement}
\begin{proof}
Let $G = \cart_{i=1}^d G_i$, $K = \cart_{i=1}^d K_{q_i}$, and $c = \max_{i \in [d]} c_i$. We label the vertices of $G_i$ using distinct elements of $[q_i]$. This defines a bijection $l: V(G) \into [q_1] \times \cdots \times [q_d]$. Henceforth, we will identify $v$ with $l(v)$, for all $v \in V(G)$. We do the same for $K$.

For a $d$-cube $C = [c_1, c_1+1] \times \cdots \times [c_d, c_d + 1] \subset \R^d$, we call the point $(c_1, \ldots, c_d)$ as the {\em origin} of $C$ and denote it by $o(C)$. A  $d$-cube is completely determined by its origin. Two cubes $C_1$ and $C_2$ intersect if and only if $d_{\infty}\big(o(C_1), o(C_2)\big) \leq 1$, where $d_{\infty}(x,y) = \max_{i \in [d]} |x(i) - y(i)|$ is the supremum norm in $\R^d$. Hence we can identify a cube representation of a graph $H$ with an embedding $f: V(H) \into \R^d$ such that $\{u, v\} \in E(H) \iff d_{\infty}(f(u), f(v)) \leq 1$. We will call $f$ as a {\em cube embedding} of $H$.

Let $f_i : V(G_i) \into \R^c$ be a cube embedding of $G_i$ for each $i \in [d]$, which exists since $\cubicity(G_i) = c_i \leq c$. Define $F : V(G) \into \R^c$ by $F((v_1, \ldots, v_d)) = f_1(v_1) + \cdots + f_d(v_d)$. Let $H$ be the graph on the vertex set $V(G)$ whose cube representation is $F$. We will show that $H \cap K = G$. Then both the assertions in the theorem will follow from Observation \ref{observationIntersection}.

It is easy to see that $K$ is a supergraph of $G$. We show that $H$ is also a supergraph of $G$. If $x,y$ are adjacent vertices in $G$, then they differ in exactly one position, say $j \in [d]$ and $\{x(j), y(j)\} \in E(G_j)$. Hence, $d_{\infty}(F(x), F(y)) = d_{\infty}(f_j(x(j)), f_j(y(j))) \leq 1$ making $x$ adjacent with $y$ in $H$. 

A pair of distinct non-adjacent vertices $x, y \in V(G)$ is called a {\em layer non-edge} if the $d$-tuples $x$ and $y$ differ in exactly one position and a {\em cross non-edge} otherwise. All the cross non-edges in $G$ are non-adjacent in $K$. We complete the proof by showing that all the layer non-edges in $G$ are non-adjacent in $H$. Let $\{x,y\}$ be a layer non-edge in $G$, i.e., $x$ and $y$ differ in only one position, say $j \in [d]$ and $\{x(j), y(j)\} \notin E(G_j)$. Hence $d_{\infty}(F(x), F(y)) = d_{\infty}(f_j(x(j)), f_j(y(j))) > 1$, and hence $x$ is not adjacent to $y$ in $H$.
\end{proof}

\subsubsection{Hypercubes}

Our upper bound on boxicity of hypercubes uses a result from the theory of partial order dimensions.

\begin{definition}[Poset dimension]
Let $(P, \lhd)$ be a poset (partially ordered set). A {\em linear extension} $L$ of $P$ is a total order which satisfies $(x \lhd y \in P) \implies (x \lhd y \in L)$. A {\em realiser} of $P$ is a set of linear extensions of $P$, say $\mathcal{R}$, which satisfy the following condition: for any two distinct elements $x$ and $y$, $x\lhd y \in P$ if and only if $x \lhd y \in L$, $\forall L \in \mathcal{R}$.  The \emph{poset dimension} of $P$, denoted by $\pdim(P)$, is the minimum positive integer $k$ such that there exists a realiser of $P$ of cardinality $k$. 
\end{definition}

Among the several consequences of the connection between boxicity and poset dimension established in \cite{DiptAdiga}, the one that we will use here is the following.

\begin{theorem}[\cite{DiptAdiga}]
\label{theoremBoxicityPoset}
Let $G$ be a bipartite graph with parts $A$ and $B$. Let $(\calP, \lhd)$ be the poset on $A \cup B$, with $a \lhd b$ if $a \in A$, $b \in B$ and $\{a,b\} \in E(G)$. Then
\[
	\half \pdim (\calP) \leq \boxicity (G) \leq \pdim (\calP).
\]
\end{theorem}

There is a natural poset associated with the hypercube called the {\em Boolean lattice}.

\begin{definition}
The $d$-dimensional {\em Boolean lattice}, denoted by $(B_d, \lhd)$, is the poset on $V(K_2^d)$ such that $u \lhd v$ if and only if $u(i) \leq v(i), ~\forall i \in [d]$. The {\em Hamming weight} of a vertex $v$ in $K_2^d$, denoted by $h(v)$, is the number of ones in $v$. The set $B_d(i) = \{v \in V(K_2^d) : h(v) = i\}, i \in \{0, \ldots, d\}$ is called the {\em $i$-th layer} of $B_d$. The subposet of $B_d$ induced on layers $i$ and $j$, $i < j$, is denoted by $B_d(i,j)$.
\end{definition}

The poset dimension of $B_d(i,j)$ for various choices of $i$ and $j$ has been a subject of extensive study starting from the study of $B_d(1,2)$ by Ben Dushnik in 1947 \cite{dushnik}. Later, Joel Spencer showed that the poset dimension of $B_d(1,2)$ is $(1 + o(d)) \log\log d$ \cite{scramble}. The current best upper bound for $B_d(j-1, j),  j \in [d]$, i.e., the subposet induced on two neighbouring layers of $B_d$, is $\order{ \log d / \log\log d }$ due to Kostochka \cite{kostochka1997dimension}. The best known lower bound for the same is $\orderatleast{ \log\log d }$, which follows from the result of Spencer mentioned above.

We use the $3 \ln d / \ln\ln d$ upper bound on $\pdim B_d(j-1, j)$ to prove the  upper bound on the boxicity of hypercubes given in Theorem \ref{theoremHypercubeBooleanLattice}. 

\begin{restatement}[Proof of Theorem \ref{theoremHypercubeBooleanLattice}]
Let $b_d$ be the largest dimension possible of a poset formed by two adjacent levels of a Boolean lattice over a universe of $d$ elements. Then
\[
\begin{array}{rcccl}
\half b_d &\leq& 
	\boxicity (K_2^d) &\leq& 
		3b_d.
\end{array}
\]
Furthermore, $\boxicity(K_2^d) \leq 12 \log d / \log\log d$.
\end{restatement}
\begin{proof}
In order not to introduce more notation, we will (ab)use the same notation for a poset and its underlying (comparability) graph. Let $H = K_2^d$.

The lower bound follows from Theorem \ref{theoremBoxicityPoset} and Observation \ref{observationInducedSubgraph} since the graph $B_d(j-1, j)$ is an induced subgraph of $K_2^d$ for all $j \in [d]$. The upper bound will be proved by showing the existence of $3$ graphs, $H_0, H_1, H_2$, each of boxicity at most $b_d$ such that $H = H_0 \cap H_1 \cap H_2$. Then the bound follows from Observation \ref{observationIntersection}.

Let $(V_0, V_1, V_2)$ be a partition of $V(H)$ such that $V_k = \{v \in V(H) : h(v) \equiv k \bmod{3}\}$. Let $H_k = H[V_{k+1} \cup V_{k+2}] \otimes V_k, ~k \in \Z_3$, where $H[S]$ denotes the subgraph of $H$ induced on $S$, and the operation $\otimes$ is as described in Observation \ref{observationUniversalVertices}. The graph $H[V_{k+1} \cup V_{k+2}]$ is a disjoint union of the graphs $B_d(j-1,j), j \in [d], j \equiv k+2 \bmod{3}$ and $B_d(0)$ and/or $B_d(d)$ in some cases. Since $B_d(0)$, $B_d(d)$, and $B_d(j-1,j)$ are bipartite graphs, by Theorem \ref{theoremBoxicityPoset}, their boxicities are at most their poset dimensions, which is at most $b_d$. Hence the boxicity of $H[V_{k+1} \cup V_{k+2}]$ is at most $b_d$ by Observation \ref{observationDisjointUnion}. Therefore, by Observation \ref{observationUniversalVertices}, the boxicity of $H_k$ is at most $b_d$ for every $k \in \Z_3$.  

We complete the proof by showing that $H = H_0 \cap H_1 \cap H_2$. It is easy to see that each $H_k, k \in \Z_3$ is a supergraph of $H$. Hence we only need to show that if $u$ and $v$ is an arbitrary pair of non-adjacent vertices in $H$, then they are non-adjacent in at least one $H_k, k \in \Z_3$. Let $k \in \Z_3 \setminus \{h(u) \bmod{3}, h(v) \bmod{3} \}$. Then $u,v \in V_{k+1} \cup V_{k+2}$ and hence they remain non-adjacent in $H_k$. 

Hence $\boxicity(K_2^d) \leq 12 \log d/ \log\log d$, by Kostochka's result.
\end{proof}

\subsubsection{Hamming graphs}

In order to extend the bounds on boxicity and cubicity of hypercubes to Hamming graphs we need to introduce some more notation. The vertices of the Hamming graph $K_q^d$ will be labelled by elements of $[q]^d$ in the natural way. Hence two vertices are adjacent if and only if their {\em Hamming distance}, i.e., the number of positions in which their labels differ, is exactly $1$. For a vertex $u$ in $K_q^d$ and for any $i \in [d]$, we shall use $u(i)$ to denote the $i$-th coordinate of the label of $u$. 

\begin{definition}[Weak Homomorphism]
\label{definitionWeakHomomorphism}
Given two graphs $G$ and $H$, a function $f : V(G) \into V(H)$ is called a {\em weak homomorphism} if for every $\{u, v\} \in E(G)$ either \\ $\{f(u),f(v)\} \in E(H)$ or $f(u) = f(v)$.
\end{definition}

\begin{remark}
If $H^o$ denotes the graph $H$ with a self-loop added at every vertex, then a weak homomorphism from $G$ to $H$ is a standard homomorphism from $G$ to $H^o$.
\end{remark}

\begin{definition}[$H$-Realiser]
A family $\calF$ of weak homomorphisms from $G$ to $H$ is called an {\em $H$-realiser} of $G$ if for every $u, v \in V(G)$ such that $\{u, v\} \notin E(G)$, there exists an $f \in \calF$ such that $f(u) \neq f(v)$ and $\{f(u), f(v)\} \notin E(H)$. If $G$ has an $H$-realiser then the cardinality of a smallest such realiser is called the {\em $H$-dimension} of $G$ and is denoted as $\hdim(G, H)$.
\end{definition}

The following lemma is an easy observation.

\begin{lemma}
\label{lemmaRealiser}
For any graph $G$, if there exists an $H$-realiser of $G$ for some graph $H$, then
\begin{eqnarray*}
\boxicity(G) & \leq & \hdim(G, H) \boxicity(H), \textnormal{ and} \\
\cubicity(G) & \leq & \hdim(G, H) \cubicity(H). 
\end{eqnarray*}
\end{lemma}

\begin{definition}
\label{definitionDoubleDistinguishing}
A family $\calD$ of sets is called {\em double distinguishing} if for any two pairs of set $A,A'$ and $B, B'$ from $\calD$, such that $A \neq A'$ and $B \neq B'$, we have
$$ (A \sdiff  A') \cap (B \sdiff B') \neq \emptyset, $$
where $A \sdiff A'$ denotes the symmetric difference of $A$ and $A'$, i.e., $(A \setminus A') \cup (A' \setminus A)$.
\end{definition}

\begin{lemma}
\label{lemmaDoubleDistinguishing}
For a set $U$, there exists a double distinguishing family $\calD$ of subsets of $U$ with $|\calD| = \floor{ c^{|U|} }$, where $c = (4/3)^{1/4}$.
\end{lemma}
\begin{proof}
Let $|U| = n$ and $q = \floor{c^n}$. Construct a family $\calD = \{S_1, \ldots, S_q\}$ of subsets of $U$ by choosing every $u \in U$ to be in $S_i$ with probability $1/2$, independent of every other choice. Given $A, A',B,B' \in \calD$, such that $A \neq A'$ and $B \neq B'$, the probability that a particular $u \in U$ is present in $(A \sdiff A') \cap (B \sdiff B')$ is at least $1/4$ (In fact, it is exactly $1/4$ when $\{A, A'\} \neq \{B, B'\}$ and $1/2$ otherwise). Hence the probability that $(A \sdiff A') \cap (B \sdiff B') = \emptyset$, i.e., the probability that no $u \in U$ goes into $(A \sdiff A') \cap (B \sdiff B')$, is at most $(3/4)^n$. So, by a union bound, the probability $p$ that $\calD$ is not double distinguishing is less than $q^4 (3/4)^n$, which is at most $1$ by our choice of $q$. Hence there exists a double distinguishing family of size $q$.
\end{proof}

Lemma \ref{lemmaDoubleDistinguishing} guarantees that we can label the alphabet $[q]$ using sets from a double distinguishing family $\calD$ of subsets of a small universe $U$ ($|U| \leq \ceil{10 \log q}$). Every element $u \in U$ defines a natural bipartition of the alphabet $[q]$ between sets that contain $u$ and those that do not. Each of those bipartitions gives a weak homomorphism from $K_q^d$ to $K_2^d$. We show that that collection of weak homomorphisms form a $K_2^d$-realiser of $K_q^d$.

\begin{lemma}
\label{lemmaHamming}
Let $K_q^d$ be the d-dimensional Hamming graph on alphabet $[q]$ and let $K_2^d$ be the d-dimensional hypercube. Then for $d \geq 2$,
\[ 
	\half \log q \leq \hdim(K_q^d, K_2^d) \leq \ceil{10 \log q}.
\]
\end{lemma}
\begin{proof}
The two-dimensional Hamming graph $K_q^2$ is an induced subgraph of $K_q^d$. It is easy to check that $K_q^2$ is isomorphic to the line graph of $K_{q,q}$, the complete bipartite graph with $q$ vertices on each part. It was shown in \cite{basavaraju2012pairwise}, that the boxicity of the line graph of $K_{q,q}$ is at least $\log q$. Hence the lower bound follows from Lemma \ref{lemmaRealiser} and the easy fact that $\boxicity K_2^2 = 2$.

Let $n = \ceil{10 \log q}$ and $U = [n]$. Since $c^{10 \log q} \geq q$, where $c= (4/3)^{1/4}$, by Lemma \ref{lemmaDoubleDistinguishing}, we know that there exists a double distinguishing family $\calD = \{S_1, \ldots, S_q\}$ in $2^U$.

Consider the family of of $n$ functions $f_u : [q] \into [2], u \in U$, where $f_u(x) = 1$ if $u \in S_x$ and $2$ otherwise. Extend each $f_u$ to a function $F_u : [q]^d \into [2]^d$ by $F_u = (f_u, \ldots, f_u)$. We claim that $\calF = \{F_u \}_{u \in U}$ is a $K_2^d$-realiser of $K_q^d$.

It is easy to see that each $F_u$ is a weak homomorphism from $K_q^d$ to $K_2^d$. Hence it suffices to show that for every pair of non-adjacent vertices $x, y \in V(K_q^d)$, there exists an $F_u, u \in U$, such that $F_u(x)$ and $F_u(y)$ are distinct and non-adjacent in $K_2^d$. Since $x$ and $y$ are non-adjacent in $K_q^d$, they are different in at least two coordinates, say $i$ and $j$, $i \neq j$. Let $A = S_{x(i)}$, $A' = S_{y(i)}$, $B = S_{x(j)}$, $B' = S_{y(j)}$. Since $i$ and $j$ are positions where $x$ and $y$ differ, we have $A \neq A'$ and $B \neq B'$. Since $\calD$ is double distinguishing, we have some $u \in U$ such that $u \in (A \sdiff A') \cap (B \sdiff B')$. Hence $f_u(x(i)) \neq f_u(y(i))$ and $f_u(x(j)) \neq f_u(y(j))$. So $F_u(x)$ and $F_u(y)$ differ in at least two coordinates and hence are distinct and non-adjacent in $K_2^d$.
\end{proof}

Completing the proof of Theorem \ref{theoremHamming} is now easy.

\begin{restatement}[Proof of Theorem \ref{theoremHamming}]
Let $K_q^d$ be the d-dimensional Hamming graph on the alphabet $[q]$ and let $K_2^d$ be the d-dimensional hypercube. Then for $d \geq 2$,
\[ 
\begin{array}{rcccl}
\log q &\leq & 
	boxicity(K_q^d) & \leq & 
	\ceil{10 \log q} \boxicity(K_2^d), \AND \\
\log q &\leq & 
	cubicity(K_q^d) & \leq & 
	\ceil{10 \log q} \cubicity(K_2^d).
\end{array}
\]
\end{restatement}
\begin{proof}
The upper bounds follows from Lemmata \ref{lemmaRealiser} and \ref{lemmaHamming}. Once we note that $K_q^2$ is isomorphic to the line graph of a complete bipartite graph, the lower bounds follow from Corollary $27$ in \cite{basavaraju2012pairwise} which is a result on boxicity of line graphs of complete bipartite graphs.

\end{proof}

\subsection{Direct products}


\begin{restatement}[Proof of Theorem \ref{theoremDirectStrong}]
For graphs $G_1, \ldots, G_d$, 
\begin{eqnarray*}
\boxicity(\direct_{i=1}^d G_i) & \leq & 
	\boxicity(\strong_{i=1}^d G_i) + 
	\boxicity(\direct_{i=1}^d K_{\chi_i})
	\textnormal{ and} \\
\cubicity(\direct_{i=1}^d G_i) & \leq & 
	\cubicity(\strong_{i=1}^d G_i) + 
	\cubicity(\direct_{i=1}^d K_{\chi_i})
\end{eqnarray*}
where $\chi_i$ denotes the chromatic number of $G_i, i \in [d]$.
\end{restatement}
\begin{proof}
Let $G_{\direct} = \direct_{i=1}^d G_i$, $G_{\strong} = \strong_{i=1}^d G_i$ and $K_{\direct} = \direct_{i=1}^d K_{\chi_i}$. Let $b_s = \boxicity(G_{\strong})$ and $b_{\chi} = \boxicity(K_{\direct})$. Furthermore, let $f_s$ and $f_{\chi}$ be $b_s$-box and $b_{\chi}$-box representations of $G_{\strong}$ and $K_{\direct}$, respectively. Finally, let $c_i : V(G_i) \into [\chi_i]$ be a proper colouring of $G_i, i \in [d]$. It is easy to see that $f$ defined by $f((v_1, \ldots, v_d)) = f_s((v_1, \ldots, v_d)) \times f_{\chi}((c_1(v_1), \ldots, c_d(v_d)))$, is a $(b_s + b_{\chi})$-box representation for $G_{\direct}$. The case for cubicity is also similar.
\end{proof}

\begin{restatement}[Proof of Theorem \ref{theoremBoxicityCompleteDirect}]
Let $q_i \geq 2$ for each $i \in [d]$. Then,
\[
\begin{array}{rcccl}
\frac{1}{2} \sum_{i=1}^d (q_i - 2) &\leq& 
	\boxicity\left(\direct_{i=1}^d K_{q_i} \right) &\leq& 
	\sum_{i=1}^d q_i, \AND \\
\frac{1}{2} \sum_{i=1}^d (q_i - 2) &\leq& 
	\cubicity\left(\direct_{i=1}^d K_{q_i} \right) &\leq& 
	\sum_{i=1}^d q_i \log (n/q_i),
\end{array}
\]
where $n = \Pi_{i=1}^d q_i$ is the number of vertices in $\direct_{i=1}^d K_{q_i}$.
\end{restatement}
\begin{proof}
Let $G = \direct_{i=1}^d K_{q_i}$ and $q = \sum_{i=1}^d q_i$. Label the vertices of each $K_{q_i}$ with $[q_i]$ so that $V(G) = [q_1] \times \cdots \times [q_d]$. 

First we show the upper bound for boxicity. Let $V_{i,j} = \{(v_1, \ldots, v_d) \in V(G) : v_i = j \}$ for $i \in [d]$ and $j \in [q_i]$ for each $i$. For each $V_{i,j}$, construct an interval graph $I_{i,j}$ in which the vertices in $V_{i,j}$ are mapped to mutually disjoint intervals on $\R$ and every other vertex is mapped to the universal interval $\R$. Next we show that $G = \cap_{i \in [d]} \cap _{j \in [q_i]} I_{i,j}$, from which the theorem follows. 

It is easy to see that each $V_{i,j}$ is an independent set since every vertex in $V_{i,j}$ have the same $i$-th component. Hence every interval graph $I_{i,j}$ is a supergraph of $G$. Let $u, v \in V(G)$ be two distinct non-adjacent vertices in $G$. Then, since they are non-adjacent, they agree in some component, say the $i$-th. Hence both the vertices belong to $V_{i,j}$ where $j$ is their common value in $i$-th component. Hence $G = \cap_{i \in [d]} \cap _{j \in [q_i]} I_{i,j}$.

Next, we show the lower bound for boxicity. For each $i \in [d]$, set 
\[
\begin{array}{rcl}
A_i &=& \{(a_1, \ldots, a_d) \in V(G) : 
	a_i > 2 \AND a_j = 1 \IF j \neq i \}, 
	\quad \textnormal{and} \\
B_i &=& \{(b_1, \ldots, b_d) \in V(G) : 
	b_i > 2 \AND b_j = 2 \IF j \neq i \}.
\end{array}
\]
 
Also set $A = \cup_{i \in [d]} A_i$ and $B = \cup_{i \in [d]} B_i$. Note that $|A_i| = |B_i| = q_i - 2$. When $d = 2$, it is easy to see that $G[A \cup B] = G[A_1 \cup B_1] \otimes G[A_2 \cup B_2]$ and that $G[A_i \cup B_i], i \in [2]$, are crown graphs, that is, a complete bipartite graph with a perfect matching removed. Hence $\boxicity(G[A_i \cup B_i]) = (q_i - 2)/2$ and hence by Observation \ref{observationJoin}, $\boxicity(G[A \cup B]) = \frac{1}{2} \sum_{i=1}^2 (q_i -2)$. When $d > 2$, $G[A \cup B]$ is a crown graph with parts $A$ and $B$. Hence $\boxicity(G[A \cup B]) = \frac{1}{2}|A| =  \frac{1}{2} \sum_{i=1}^d (q_i -2)$. In either case, the lower bound now follows easily from Observation \ref{observationInducedSubgraph}.

Each of the interval graph $I_{i,j}$ can be represented as an intersection graph of \\ $\ceil{\log |V_{i,j}|}=\ceil{\log \Pi_{k \neq i} q_k}$ unit interval graphs. Hence the upper bound on cubicity. The lower bound on  cubicity follows since it cannot be lower than the boxicity.
\end{proof}

\end{appendix}

\begin{thebibliography}{10}

\bibitem{DiptAdiga}
Abhijin Adiga, Diptendu Bhowmick, and L.~Sunil Chandran.
\newblock Boxicity and poset dimension.
\newblock In {\em COCOON}, pages 3--12, 2010.

\bibitem{basavaraju2012pairwise}
Manu Basavaraju, L~Sunil Chandran, Rogers Mathew, and Deepak Rajendraprasad.
\newblock Pairwise suitable family of permutations and boxicity.
\newblock {\em arXiv preprint arXiv:1212.6756}, 2012.

\bibitem{chintan}
L.~Sunil Chandran, Anita Das, and Chintan~D. Shah.
\newblock Cubicity, boxicity, and vertex cover.
\newblock {\em Discrete Mathematics}, 309(8):2488--2496, 2009.

\bibitem{CMO}
L.~Sunil Chandran, C.~Mannino, and G.~Orialo.
\newblock On the cubicity of certain graphs.
\newblock {\em Information Processing Letters}, 94:113--118, 2005.

\bibitem{RogSunSiv}
L.~Sunil Chandran, Rogers Mathew, and Naveen Sivadasan.
\newblock Boxicity of line graphs.
\newblock {\em Discrete Mathematics}, 311(21):2359--2367, 2011.

\bibitem{CN99}
L.~Sunil Chandran and Naveen Sivadasan.
\newblock The cubicity of hypercube graphs.
\newblock {\em Discrete Mathematics}, 308(23):5795--5800, 2008.

\bibitem{dushnik}
B.~Dushnik.
\newblock Concerning a certain set of arrangements.
\newblock {\em Proceedings of the American Mathematical Society},
  1(6):788--796, 1950.

\bibitem{imrich2011handbook}
Richard Hammack, Wilfried Imrich, and Sandi Klav{\v{z}}ar.
\newblock {\em Handbook of product graphs}.
\newblock CRC press, 2011.

\bibitem{kostochka1997dimension}
AV~Kostochka.
\newblock The dimension of neighboring levels of the boolean lattice.
\newblock {\em Order}, 14(3):267--268, 1997.

\bibitem{Roberts}
F.~S. Roberts.
\newblock {\em Recent progresses in combinatorics}, chapter On the boxicity and
  cubicity of a graph, pages 301--310.
\newblock Academic Press, New York, 1969.

\bibitem{scramble}
J.~Spencer.
\newblock Minimal scrambling sets of simple orders.
\newblock {\em Acta Mathematica Hungarica}, 22:349--353, 1957.

\bibitem{west2008boxicity}
Douglas~B. West.
\newblock Boxicity and maximum degree.
\newblock \url{http://www.math.uiuc.edu/~west/regs/boxdeg.html}, 2008.
\newblock Accessed: Jan 12, 2013.

\end{thebibliography}
\end{document}